\documentclass[1p]{elsarticle}
\usepackage{amsthm,amsmath,amssymb}
\usepackage{stmaryrd,mathrsfs}

\makeatletter
\def\ps@pprintTitle{%
 \let\@oddhead\@empty
 \let\@evenhead\@empty
 \def\@oddfoot{}%
 \let\@evenfoot\@oddfoot}
\makeatother

\newcommand{\vertiii}[1]{{\left\vert\kern-0.25ex\left\vert\kern-0.25ex\left\vert #1
    \right\vert\kern-0.25ex\right\vert\kern-0.25ex\right\vert}}

\newcommand{\supp}[0]{\operatorname{supp}}

\newcommand{\R}{\mathbb{R}}

\newcommand{\F}{\mathcal{F}}
\newcommand{\f}{\frac}
\newcommand{\cm}{\mathcal M}

\theoremstyle{plain}
\newtheorem{theorem}[equation]{Theorem}
\newtheorem{proposition}[equation]{Proposition}
\newtheorem{corollary}[equation]{Corollary}
\newtheorem{lemma}[equation]{Lemma}

\theoremstyle{definition}

\newtheorem{remark}[equation]{Remark}
\newtheorem{example}[equation]{Example}

\numberwithin{equation}{section}

\begin{document}

\begin{frontmatter}


\title{Mixed weak type estimates: Examples and counterexamples related to  a problem of E. Sawyer}

\author[ehu,Ikerbasque,BCAM]{Carlos P\'erez\fnref{Fund1}}
\ead{carlos.perezmo@ehu.es}
\author[UNS]{Sheldy Ombrosi\fnref{Fund2}}
\ead{sombrosi@uns.edu.ar}

\fntext[Fund1]{
This research is supported by the Basque Government through the BERC 2014-2017 program and by Spanish Ministry of Economy and Competitiveness MINECO: BCAM Severo Ochoa excellence accreditation SEV-2013-0323 and the project MTM2014-53850-P.}

\address[ehu]{Department of Mathematics, University of the Basque Country UPV/EHU, Bilbao, Spain} 
\address[Ikerbasque]{IKERBASQUE, Basque Foundation for Science, Bilbao, Spain}
\address[BCAM]{BCAM, Basque Center for Applied Mathematics, Bilbao, Spain}
\address[UNS]{Department of Mathematics, Universidad Nacional del Sur, 
Bah\'ia Blanca, Argentina.} 

\begin{abstract}
We study mixed weighted weak-type inequalities for families of functions,
which can be applied to study classical operators in harmonic analysis
Our main theorem extends the key
result from \cite{CMP2}.
\end{abstract}



\end{frontmatter}





\section{Introduction and main results}

In this work we consider mixed weighted weak-type inequalities of
the form
\begin{equation}\label{desigualdad}
 uv\bigg(\bigg\{x\in \R^n : \frac{|T(fv)(x)|}{v(x)} >t \bigg\}\bigg)
\leq \frac{C}{t}\int_{\R^n} |f(x)|\,Mu(x)v(x)\,dx,
\end{equation}
where  $T$ is either the
Hardy-Littlewood maximal operator or any Calder\'on-Zygmund
operator. Similar inequalities were studied by
Sawyer in \cite{Sa} motivated by the work of Muckenhoupt and
Wheeden \cite{MW} (see also the works \cite{AM} and \cite{MOS}).

E. Sawyer proved that inequality \eqref{desigualdad} holds in
$\R$ when $T=M$ is the Hardy-Littlewood maximal operator assuming that the
weights $u$ and $v$ belong to the class $A_1$.  This
result can be seen as a very delicate extension of the classical
weak type $(1,1)$ estimate. However, the reason why E. Sawyer considered
\eqref{desigualdad} is due to the following interesting
observation. Indeed, inequality \eqref{desigualdad} yields a new proof of the classical
Muckenhoupt's theorem for $M$ assuming that the $A_p$ weights
can be factored (P. Jones's theorem). This means that if $w\in A_p$ then
$w=uv^{1-p}$ for some $u,v \in A_1$.  Now, define the operator
 $f\rightarrow \frac{M(fv)}{v}$ which is bounded on
$L^{\infty}(uv)$ and it is of weak type $(1,1)$ with respect to the measure $uvdx$  by \eqref{desigualdad}. Hence by  the Marcinkiewicz  interpolation theorem we recover
Muckenhoupt's theorem.

In the same paper, Sawyer conjectured that if $T$ is instead the
Hilbert transform the inequality also holds with the same
hypotheses on the weights $u$ and $v$.  This conjecture was proved
in \cite{CMP2}. In fact, it is proved in this paper that the
inequality \eqref{desigualdad} holds for both the Hardy-Littlewood
maximal operator and for any Calder\'on-Zygmund Operator in any
dimension if either the weights $u$ and $v$ both belong to $A_1$ or
$u$ belongs to $A_1$ and $uv \in A_{\infty}$. The method of proof
is quite different from that in  \cite{Sa} (also from \cite{MW})
and it is based on certain ideas from extrapolation that go back
to the work of Rubio de Francia (see \cite{CMP2} and also the
expository paper \cite{CMP3}).
Applications of these results can be found in \cite{LOPTT}. The authors
conjectured in \cite{CMP2} that their results may hold under weaker hypotheses on the weights. To
be more precise,  they proposed that inequality \eqref{desigualdad}
is true if $u\in A_{1}$ and $v\in A_{\infty}$. Very recently, some quantitative estimates in terms of the relevant constants of the weights have been obtained  in \cite{OPR} and some new conjectures have been formulated.

Inequalities like  \eqref{desigualdad}, when $T$ is the Hardy-Littlewood maximal operator, can also be seen
as generalizations of the classical Fefferman-Stein inequality
$$
\left\|  M(f)  \right\|_{L^{1, \infty}(u)} \leq c\;\|f\|_{L^1(Mu)},
$$
where $c$ is a dimensional constant.  However, in Section 3, we will see that
\eqref{desigualdad} does not hold in general even for weights satisfying strong conditions like $v\in
RH_{\infty} \subset A_{\infty}$.

In this work we generalize the extrapolation result in
\cite{CMP3} for a larger class of weights (see Theorem
\ref{theor:extrapol} below). This method of extrapolation is
flexible enough with scope reaching beyond the classical linear operators. Indeed, it can be applied to square
functions, vector valued operators as well as multilinear
singular integral operators. See Section \ref{applications} for
some of these applications. In fact,  the best way to state the extrapolation theorem is without
considering operators and the result can be seen  as a property of
families of functions. Hereafter, $\F$ will denote a family of
ordered pairs of non-negative, measurable functions $(f,g)$. Also
we are going to assume that this family $\F$ of functions,
satisfies the following property: for \textbf{some} $p_0$,
$0<p_0<\infty$, and every $w\in A_\infty$,
\begin{equation} \label{extrapol1}
 \int_{\R^n} f(x)^{p_0} w(x)\,dx \leq
C\int_{\R^n} g(x)^{p_0} w(x)\,dx,
\end{equation}
for all $(f,g)\in\F$ such that the left-hand side is finite, and
where $C$ depends only on the $A_\infty$ constant of $w$. By the
main theorem in \cite{CMP1}, this assumption turns out to be true for {\bf any} exponent $p\in(0,\infty)$ and \textbf{every}
$w\in A_\infty$,
\begin{equation} \label{extrapol2}
 \int_{\R^n} f(x)^p w(x)\,dx \leq
C\int_{\R^n} g(x)^p w(x)\,dx,
\end{equation}
for all $(f,g)\in\F$ such that the left-hand side is finite, and
where $C$ depends only on the $A_\infty$ constant of $w$. See
the papers \cite{CMP1}, \cite{CGMP} and \cite{CMP3} for more
information  and applications and the book \cite{CMP4} for a general account. It is also interesting that both \eqref{extrapol1} and
\eqref{extrapol2} are equivalent to the following vector-valued
version: for all $0<p,q<\infty$ and for all $w\in A_\infty$ we
have
\begin{eqnarray}
\Big\| \Big(\sum_j (f_j)^q\Big)^{\frac{1}{q}} \Big\|_{L^{p}(w)}
&\le& C\, \Big\| \Big(\sum_j (g_j)^q\Big)^{\frac{1}{q}}
\Big\|_{L^{p}(w)},
\label{CMP:v-v}
\end{eqnarray}
for any $\{(f_j,g_j)\}_j\subset \F$, where these estimates hold
whenever the left-hand sides are finite.

Next theorem improves the corresponding Theorem from \cite{CMP2}.

\begin{theorem} \label{theor:extrapol}
Let $\F$ be a family of functions satisfying \eqref{extrapol1} and
let $\theta\geq 1$. Suppose that $u\in A_1$ and that $v$ is a
weight such that for some $\delta>0$, $v^{\delta} \in A_{\infty}$.

Then, there is a constant $C$ such that
\begin{equation} \label{extrapolweak}
 \Big\| \frac{f}{v^{\theta} }\Big\|_{L^{1/\theta,\infty}(uv)}
 \leq
 C\,\Big\| \frac{g}{v^{\theta} }\Big\|_{L^{1/\theta,\infty}(uv)}
, \qquad (f,g)\in\F.
\end{equation}

Similarly, the following vector-valued extension holds: if \,$0<q<\infty$,
\begin{equation}\label{Lp,s:v-v}
 \Big\| \frac{\sum_j (f_j)^q \Big)^\frac1q}{v^{\theta}}\Big\|_{L^{1/\theta,\infty}(uv)}
 \leq
 C\,\Big\| \frac{\sum_j (g_j)^q \Big)^\frac1q}{v^{\theta}
 }\Big\|_{L^{1/\theta,\infty}(uv)},
\end{equation}
for any $\{(f_j,g_j)\}_j\subset \F$.
\end{theorem}

Observe that the singular class of weights $v(x)=|x|^{-nr}$, $r\geq1$, is covered by  the hypothesis of the Theorem  but not in the
corresponding Theorem from \cite{CMP2}.

The proof of \eqref{Lp,s:v-v} is immediate since we can
extrapolate using as initial hypothesis \eqref{CMP:v-v} and then applying
\eqref{extrapolweak}.

\begin{corollary} \label{corol:extrapol}
Let $\F$, $u$ and $\theta\geq 1$ as in the Theorem. Suppose now
that $v_i$, $i=1, \cdots, m$, are weights such that for some
$\delta_i>0$, $v_i^{\delta_i} \in A_{\infty}$, $i=1, \cdots, m$.

Then, if we denote $v=\prod_{i=1}^{m} v_i$
\begin{equation*}
 \Big\| \frac{f}{v^{\theta} }\Big\|_{L^{1/\theta,\infty}(uv)}
 \leq
C\, \Big\| \frac{g}{v^{\theta} }\Big\|_{L^{1/\theta,\infty}(uv)} ,
\qquad (f,g)\in\F.
\end{equation*}
and similarly for \, $0<q<\infty$,
\begin{equation*}
 \Big\| \frac{\sum_j (f_j)^q \Big)^\frac1q}{v^{\theta}}\Big\|_{L^{1/\theta,\infty}(uv)}
 \leq
 C\,\Big\| \frac{\sum_j (g_j)^q \Big)^\frac1q}{v^{\theta}
 }\Big\|_{L^{1/\theta,\infty}(uv)},
\end{equation*}
for any $\{(f_j,g_j)\}_j\subset \F$.

\end{corollary}

The proof reduces to the Theorem by choosing $\delta>0$ small
enough such that $v^{\delta}=\prod_{i=1}^{m} v_i^{\delta} \in
A_{\infty}$ which follows by convexity since $v_i^{\delta_i} \in
A_{\infty}$, $i=1, \cdots, m$.

To apply Theorem \ref{theor:extrapol} above to some of the classical
operators we need a mixed weak type estimate for the
Hardy-Littlewood maximal operator.  This is the content of next Theorem which was obtained in dimension one by Andersen and
Muckenhoupt in \cite{AM} and by Mart{\'\i}n-Reyes, Ortega Salvador and
Sarri{\'o}n Gavi{\'a}n \cite{MOS} in higher dimensions. In each case the
proof follows as a consequence of a more general result with the
additional hypothesis that $u\in A_1$.  For completeness we will give an independent and direct proof with the advantage that no
condition on the weight $u$ is assumed.

\begin{theorem}\label{theor:max}
Let $u\ge 0$ and $v(x)=|x|^{-nr}$ for some $r>1$. Then there is a
constant $C$ such that for all $t>0$,
\begin{equation} \label{max1}
 uv\bigg(\bigg\{x\in \R^n : \frac{M(fv)(x)}{v(x)} >t \bigg\}\bigg)
\leq \frac{C}{t}\int_{\R^n} |f(x)|\,Mu(x)v(x)\,dx.
\end{equation}
\end{theorem}

\begin{remark}\label{falso}\rm
We remark that the theorem could be false when $r=1$ even in the case $u=1$, see \cite{AM}. However,
we already mentioned that the singular weight  $v(x)=|x|^{-n}$ is included in the extrapolation Theorem
\ref{theor:extrapol}.
\end{remark}

\vskip 2mm \noindent {\bf Acknowledgement.} The authors are grateful
to F. J. Mart\'{i}n-Reyes and P. Ortega-Salvador to point out
reference \cite{MOS}.

\section{Some applications}\label{applications}

In this section we show the flexibility of the method by giving
two applications.

\subsection{The vector-valued case}

Let $T$ be any singular integral operator with standard kernel and
let $M$ be the Hardy-Littlewood maximal function. We are going to
show that starting from the following inequality due to Coifman
\cite{Coi}: for $0<p<\infty$ and $w\in A_\infty$,
\begin{equation}\label{coifman}
\int_{\R^n} |Tf(x)|^p\, w(x)\,d x \le C\,\int_{\R^n} Mf(x)^p\,
w(x)\,d x,
\end{equation}
 combined with the extrapolation Theorem \ref{theor:extrapol}
together with Theorem \ref{theor:max} yields the following
corollary.

\begin{corollary}\label{corol:vv}
Let $u\in A_1$ and $v(x)=|x|^{-nr}$ for some $r>1$. Also let
$1<q<\infty$. Then, there is a constant $C$ such that for all
$t>0$,
\begin{eqnarray*}
uv \bigg(\bigg\{ x\in \R^n : \frac{\Big(\sum_j
M(f_{j}v)(x)^q\Big)^{\frac1q}}{v(x)} >t \bigg\}\bigg)\!\!\! &\leq&
\frac{C}{t}\int_{\R^n} \Big(\sum_j |f_{j}(x)|^q\Big)^\frac1q
\,u(x)v(x)\,dx, \label{max-vv}
\\
uv \bigg(\bigg\{ x\in \R^n : \frac{\Big(\sum_j
|T(f_{j}v)(x)|^q\Big)^{\frac1q}}{v(x)}
>t
\bigg\}\bigg)\!\!\! &\leq& \frac{C}{t}\int_{\R^n} \Big(\sum_j
|f_j(x)|^q\Big)^\frac1q \,u(x)v(x)\,dx.  \label{T-vv}
\end{eqnarray*}
\end{corollary}

Observe that in particular we have the following scalar version,
$$
 uv\bigg(\bigg\{x\in \R^n : \frac{|T(fv)(x)|}{v(x)} >t \bigg\}\bigg)
\leq \frac{C}{t}\int_{\R^n} |f(x)|\,u(x)v(x)\,dx.
$$
This scalar version was proved in \cite{MOS}.

The second inequality of the corollary follows from
the first one by applying inequality \eqref{Lp,s:v-v} in Theorem
\ref{theor:extrapol} with initial hypothesis \eqref{coifman}:
%
$$\sup_{t>0}t uv \bigg(\bigg\{ x\in \R^n : \frac{\Big(\sum_j
|T(f_j)(x)|^q\Big)^{\frac1q}}{v(x)} >t \bigg\}\bigg)\le \\
$$
$$
C \sup_{t>0}t uv \bigg(\bigg\{ x\in \R^n : \frac{\Big(\sum_j
M(f_j)(x)^q\Big)^{\frac1q}}{v(x)} >t \bigg\}\bigg).
$$
%


To prove the first inequality in Corollary \ref{corol:vv} we first
note that in \cite{CGMP} it was shown for $1<q<\infty$ and for all
$0<p<\infty$ and $w\in A_\infty$,
$$
\Big\|\Big(\sum_j (M(f_j))^q\Big)^\frac1q\Big\|_{L^p(w)} \le C\,
\Big\| M  \Big( \big(\sum_j |f_j|^q \big)^\frac1q  \Big)
\Big\|_{L^p(w)}.
$$
To conclude we apply Theorem \ref{theor:extrapol} combined with
Theorem \ref{theor:max}.

\subsection{Multilinear Calder\'on-Zygmund operators:}
We now apply our main results to multilinear Calde\-r\'on-Zygmund
operator. We follow here the theory developed by Grafakos and
Torres in \cite{GT1}, that is, $T$ is an $m$-linear operator such
that $T:L^{q_1}\times\cdots\times L^{q_m} \longrightarrow L^q$,
where $1< q_1,\dots,q_m<\infty$, $0<q<\infty$ and
\begin{equation}\label{exponents}
\frac1q=\frac1{q_1}+\cdots+\frac1{q_m}.
\end{equation}
The operator $T$ is associated with a Calder\'on-Zygmund kernel
$K$ in the usual way:
$$
T (f_1,\dots, f_m)(x) = \int_{\R^n}\cdots \int_{\R^n}
K(x,y_1,\dots, y_m)\, f_1(y_1)\dots f_m(y_m)\,dy_1\dots dy_m,
$$
whenever $f_1,\dots,f_m$ are in $C_0^\infty$ and $x\notin
\bigcap_{j=1}^m\supp f_j$. We assume that $K$ satisfies the
appropriate decay and smoothness conditions (see \cite{GT1}
for complete details). Such an operator $T$ is bounded on any
product of Lebesgue spaces with exponents $1<q_1,\dots,
q_m<\infty$, $0<q<\infty$ satisfying \eqref{exponents}. Further,
it also satisfies weak endpoint estimates when some of the $q_i$'s
are equal to one. There are also weighted norm inequalities for
multi-linear Calder\'on-Zygmund operators; these were first proved
in \cite{GT2} using a good-$\lambda$ inequality and fully characterized in \cite{LOPTT} using the sharp maximal function $\cm$ and a new
maximal type function which plays a central role in the theory,
\begin{equation*}\label{first11}
\cm (f_1,\dots , f_m)(x)=\sup_{\substack{Q\ni x\\
Q\,\,\textup{cube}} } \prod_{i=1}^m\frac{1}{|Q|}\int_Q|f_i(z)|\,
dz,
\end{equation*}
where the supremum is taken over cubes with sides parallel to the axes.
Indeed, one of the main results from \cite{LOPTT} is that for
any $0<p<\infty$ and for any $w\in A_\infty$,
\begin{equation*}\label{GT2:T-M}
\Big\|T(f_1,\dots,f_m)\Big\|_{L^p(w)} \le C\, \Big\|   \cm (f_1,\dots , f_m)  \Big\|_{L^p(w)}.
\end{equation*}

Beginning with these inequalities, we can apply Theorem
\ref{theor:extrapol} to the family \\$ \F\Big(T(f_1,\dots,f_m),\,
\cm (f_1,\dots , f_m)  \Big).$ Hence, if $u\in A_1$ and
$v(x)=|x|^{-nr}$ for some $r>1$.
\begin{equation}\label{multi:v-v-strong}
\Big\|  \frac{ T(f_1,\dots,f_m) }{v^m}
\Big\|_{L^{1/m,\infty}(uv)} \le C\, \Big\| \frac{  \cm (f_1,\dots , f_m)   }{v^m} \Big\|_{L^{1/m,\infty}(uv)}
\end{equation}

\begin{corollary}\label{corol:mult-v-v}
Let $T$ be a multilinear Calder\'on-Zygmund operator as above. Let
$u\in A_1$ and $v(x)=|x|^{-nr}$ for some $r>1$. Then
$$
\Big\|  \frac{ T(f_1,\dots,f_m) }{v^m}
\Big\|_{L^{1/m,\infty}(uv)} \le C\,\prod_{j=1}^m \int_{\R^n}
|f_j|\,u\,dx,. $$

\end{corollary}

To prove this corollary we will use the following version of the
generalized Holder's inequality: for $1\leq q_1, \dots, q_m <
\infty$ with
$$\f{1}{q_1}+\dots +\f{1}{q_m}=\f{1}{q},
$$
there is a constant $C$ such that
$$\| \prod_{j=1}^m h_j \|_{ L^{q,\infty}(w)} \leq C\,
\prod_{j=1}^m \|h_j\|_{L^{ q_j,\infty}(w)}. $$
The proof of this inequality follows in a similar way that the
proof of the classic generalized Holder's inequality in $L^{p}$
theory.

Now, if we combine this with \eqref{multi:v-v-strong} and with the trivial observation that
\begin{equation*}\label{e4}
\cm(f_1,\dots , f_m)(x) \le \prod_{i=1}^m M(f_i) \, ,
\end{equation*}

we
have
$$
\Big\|  \frac{ T(f_1,\dots,f_m) }{v^m}
\Big\|_{L^{1/m,\infty}(uv)} \le C\, \prod_{j=1}^m \Big\| \frac{ M
f_j}{v} \Big\|_{ L^{1,\infty}(uv)},
$$
Finally, an application of  Theorem \ref{theor:max} concludes the proof of the
corollary.

\section{counterexamples}\label{contraejemplo}

An interesting point from
Theorem \ref{theor:max} is that if $v(x)=|x|^{-nr}$, $r>1$, the estimate
\begin{equation}\label{desigualdad2}
 uv\bigg(\bigg\{x\in \R^n : \frac{M(fv)(x)}{v(x)} >t \bigg\}\bigg)
\leq \frac{C}{t}\int_{\R^n} |f(x)|\,Mu(x)v(x)\,dx,
\end{equation}
holds for any \,$u\ge 0$. On the other hand we have already mentioned that the same inequality holds
if $u \in A_{1}$ and $v\in A_1$ or $uv\in A_{\infty}$ \cite{CMP2}. In particular, this is the case if $u\in A_{1}$ and
$v\in RH_{\infty}$.  Assuming that $v\in RH_{\infty}$, a natural
question is whether inequality \eqref{desigualdad2} holds with {\bf no} assumption on $u$. This would
improve the classical
Fefferman-Stein inequality. However, we will show in the next example
that this is \textbf{false} in general.

\begin{example}
On the real line we let $v(x)=\sum_{k\in Z}\left| x-k\right| \chi _{ I_{k}}\left(
x\right) $, where $I_{k}$ denotes the interval $\left| x-k\right|
\leq 1/2$.  It is not difficult to see that $v\in RH_{\infty }$. If
we choose
$$u(x)=\sum
_{\substack{ k\in N \\ k>10}}\frac{k}{\log(k)}\chi _{ J_{k}}\left( x\right),$$
where  $J_{k}=
\left[ k+\frac{1}{4k},k+\frac{1}{k}\right] $, and $f=$ $\chi _{ \left[ -1,1
\right] },$
then there is no finite constant $C$ such
that the inequality
\begin{equation}
uv(\left\{ x:Mf\left( x\right) >v(x)\right\} )\leq C\int \left|
f\right| M^{2}u \label{falsa}
\end{equation}
holds. To prove this we will make use of the following observation:
\begin{quote} There is a geometric constant such that
$$
M^{2}w(x) \approx M_{ L\log L}w(x) \qquad x \in \R^n
$$
\end{quote}
where
$$ M_{ L\log L}f(x) = \sup_{Q\ni x}
\|f\|_{ L\log L,Q}$$
and
$$
\|f\|_{L\log L,Q} =\inf\{\lambda >0:
\frac{1}{|Q|}\int_{Q} \Phi(\frac{ |f|}{ \lambda }) \, dx \le 1\}.
$$
with $\Phi(t)=t\log(e+t),$
see \cite{PW} or \cite{G}. Now, it is a computation to see that if
$x\in \lbrack -1,1]$, $M^{2}u(x)\approx M_{ L\log
L}u(x) \leq C$ then the right hand side of  $\left(
\ref{falsa}\right)$ is finite, while the left hand side is infinite. Let us check that. For  $\left| x\right| >2$ we have that $%
Mf\left( x\right) \geq $ $\frac{1}{\left| x\right| }$ and if $x\in
J_{k}\subset I_{k}$ for $k>10$ $\frac{1}{\left| x\right|
}>\frac{1}{2k}$, then it is easy to see that
$(k+\frac{1}{4k},k+\frac{1}{2k})\subset \{x\in J_{k}:$
$Mf(x)>v(x)\}$ and therefore we obtain that
\begin{eqnarray*}
uv(\left\{ x:Mf\left( x\right) >v(x)\right\} ) &>&\sum_{\substack{ k\in N \\ %
k>10}}\frac{k}{\log(k)}\int_{k+\frac{1}{4k}}^{k+\frac{1}{2k}}\left( x-k\right) dx> \\
&>&\sum_{\substack{ k\in N \\ k>10}}\frac{1}{8k \log(k)}=\infty .
\end{eqnarray*}

\end{example}

\section{Proof of Theorem \ref{theor:extrapol}}

The following Lemmas will be useful:

\begin{lemma} \label{lem:epsilon}
If $u\in A_1$, $w\in A_1$, then there exists $0<\epsilon_0<1$
depending only  on $[u]_{A_1}$ such that $uw^\epsilon \in A_1$ for
all $0<\epsilon<\epsilon_0$.

\end{lemma}

\begin{proof}
Since $u\in A_1$, $u\in RH_{s_0}$ for some $s_0>1$ depending on
$[u]_{A_1}$.  Let $\epsilon_0=1/{s_0}'$ and
$0<\epsilon<\epsilon_0$. This implies that $u\in RH_s$ with
$s=(1/\epsilon)'$.

Then since $u,\,v\in A_1$, for any cube $Q$ and almost every $x\in
Q$,
\begin{eqnarray*}
\lefteqn{\hskip-.3cm \frac{1}{|Q|}\int_Q u(y)w(y)^\epsilon\,dy
 \leq \left(\frac{1}{|Q|}\int_Q u(y)^s\,dy\right)^{1/s}
\left(\frac{1}{|Q|}\int_Q w(y)\,dy\right)^{1/s'}}
\\
&\leq& \frac{[u]_{RH_s}}{|Q|}\int_Q u(y)\,dy
\left(\frac{1}{|Q|}\int_Q w(y)\,dy \right)^{1/s'} \leq
[u]_{RH_s}[u]_{A_1}[w]_{A_1}^\epsilon u(x)w(x)^\epsilon.
\end{eqnarray*}%
Hence $uw^\epsilon \in A_1$ with $[uw^\epsilon]_{A_1}\le
[u]_{RH_s}[u]_{A_1}[w]_{A_1}^\epsilon$.

\end{proof}

We also need the following version of the  Marcinkiewicz
interpolation theorem in the scale of Lorentz spaces. In fact we
need a version of this theorem with precise constants.  The proof
can be found in \cite{CMP2}.

\begin{proposition}\label{prop:interpolation} Given
$p_0$, $1<p_0<\infty$, let $T$ be a sublinear operator such that
$$
\|T f\|_{L^{p_0,\infty}} \le C_0\,\|f\|_{L^{p_0,1}} \qquad \text{
and } \qquad \|T f\|_{L^{\infty}} \le C_1\,\|f\|_{L^{\infty}}.
$$
Then for all $p_0<p<\infty$,
$$
\|T f\|_{L^{p,1}} \le
2^{1/p}\,\big(C_0\,(1/p_0-1/p)^{-1}+C_1\big)\,\|f\|_{L^{p,1}}.
$$
\end{proposition}

Fix $u\in A_1$ and $v$ such that $v^{\delta} \in A_{\infty}$ for
some $\delta>0$. Then by the factorization theorem $v^{\delta}=v_1
v_2$ for some $v_1 \in A_1$  and $v_2 \in RH_{\infty}$.  Define
the operator $S_{\lambda}$ by
$$ S_{\lambda}f(x)= \frac{M(fuv_1^{1/\lambda\delta})}{uv_1^{1/\lambda\delta}}
$$
for some large enough constant $\lambda>1$ that will be chosen
soon.

By Lemma \ref{lem:epsilon}, there exists $0<\epsilon_0<1$ (that
depends only on $[u]_{A_1}$) such that $u\,w^{\epsilon}\in A_1$
for all $w\in A_1$ and $0<\epsilon<\epsilon_0$.

Choose $\lambda>\frac{1}{\delta \epsilon_0}$  such that
$uv_1^{1/\lambda\delta} \in A_1$. Hence, $S_{\lambda}$ is bounded
on $L^{\infty}(uv)$ with constant $C_1=[u]_{A_1}$. We will now
show that for some larger $\lambda$, $S_{\lambda}$ is bounded on
$L^{m}(uv)$. Observe that
$$
\int_{\R^n}\! Sf(x)^{\lambda}\, u(x)\,v(x)\,dx = \int_{\R^n}\!
M(fuv_1^{1/\lambda\delta})(x)^{\lambda}\,u(x)^{1-\lambda}\,v_2(x)^{1/\delta}
\,dx.
$$
Since $v_2=\tilde{v}_2^{1-t}$  for some $\tilde{v}_2 \in A_{1}$
and $t>1$ we have
$$u^{1-\lambda}\,v_2^{1/\delta}= u^{1-\lambda}\,\tilde{v}_2^{\frac{1-t}{\delta}}=
\big(u\, \tilde{v}_2^{\frac{t-1}{\delta(\lambda-1)}}
\big)^{1-\lambda}.
$$
By Lemma \ref{lem:epsilon} there exists $\lambda$ sufficiently large ($\lambda>1+\frac{t-1}{\delta \epsilon_0}$) such that $
u\, \tilde{v}_2^{\frac{t-1}{\delta(\lambda-1)}} \in A_1$ and hence
$u^{1-\lambda}\,v_2^{1/\delta} \in A_{\lambda}$. By Muckenhoupt's
theorem, $M$ is bounded on
$L^{\lambda}(u^{1-\lambda}v_2^{1/\delta} )$ and therefore $S$ is
bounded on $L^{\lambda}(uv)$ with some constant $C_0$. Observe
that $\lambda$ depends on the $A_1$ constant of $u$. We fix one
such $\lambda$ from now on.

By Proposition \ref{prop:interpolation} above we have that
$S$ is bounded on $L^{q,1}(uv)$, $q>\lambda$. Hence,
$$
\|S f\|_{L^{q,1}(uv)} \le
2^{1/q}\,\big(C_0\,(1/\lambda-1/q)^{-1}+C_1)\,\|f\|_{L^{q,1}(uv)}.
$$
Thus, for all $q\ge 2\lambda$ we have that $\|S
f\|_{L^{q,1}(uv)}\le K_0\,\|f\|_{L^{q,1}(uv)}$ with
$K_0=4\lambda\,(C_0+C_1)$. We emphasize that the constant $K_0$ is
valid for every $q\ge 2\lambda$.

Fix $(f,g)\in\F$ such that the left-hand side of
\eqref{extrapolweak} is finite. We let $r$ be such that  $\theta<r<\theta
(2\lambda)'$, to be chosen soon. Now, by the duality
of $L^{r,\infty}$ and $L^{r',1}$,
$$
\big\|f\,v^{-\theta}\big\|_{L^{1/\theta,\infty}(uv)}^\frac{1}{r} =
\big\| (f\,v^{-\theta})^\frac{1}{r}
\big\|_{L^{r/\theta,\infty}(uv)} = \sup \int_{\R^n}
f(x)^{\frac{1}{r}}\,h(x)\,u(x)\,v(x)^{1-\theta/r}\,dx,
$$
where the supremum is taken over all non-negative $h \in
L^{(\frac{r}{\theta})',1}(uv)$ with
$\|h\|_{L^{(\frac{r}{\theta})',1}(uv)}=1$.  Fix such a function
$h$. We are going to build a larger function $\R h$ using the
Rubio de Francia`s method such $\R h\, uv^{1-\theta/r}\in
A_\infty$. Hence we will use the hypothesis \eqref{extrapol2} with
$p=\theta/r$ (recall that is equivalent to \eqref{extrapol1}) with
the weight $\R h\, uv^{1-\theta/r}\in A_\infty$

We let $r$ be such that $(\frac{r}{\theta})'> 2\lambda$ and hence
$S_{(\frac{r}{\theta})'}$ is bounded on
$L^{(\frac{r}{\theta})',1}(uv)$ with constant bounded by $K_0$.
Now apply the Rubio de Francia algorithm (see \cite{GCRdF}) to
define the operator $\R$ on $ h\in L^{(\frac{r}{\theta})',1}(uv)$,
$h\geq 0$, by
$$
\R h(x) = \sum_{j=0}^\infty \frac{S_{(\frac{r}{\theta})'}^j
h(x)}{2^j\, K_0^j},
$$
Recall that the operator $S_{(\frac{r}{\theta})'}$ is defined by
$$ S_{(\frac{r}{\theta})'}f(x)= \frac{M(fuv_1^{1/(\frac{r}{\theta})'\delta})}{uv_1^{1/(\frac{r}{\theta})'\delta}}.
$$
Also, recall that  by the choice of $r$
$uv_1^{1/(\frac{r}{\theta})'\delta} \in A_1$.

It follows immediately from this definition that:
\begin{list}{$(\theenumi)$}{\usecounter{enumi}\leftmargin=1.4cm
\labelwidth=.7cm\labelsep=0.25cm \itemsep=0.2cm \topsep=.1cm
\renewcommand{\theenumi}{\alph{enumi}}}

\item $h(x)\le \R h(x)$;

\item $\|\R h\|_{L^{(\frac{r}{\theta})',1}(uv)}\le
2\,\|h\|_{L^{(\frac{r}{\theta})',1}(uv)}$;

\item $S_{(\frac{r}{\theta})'}(\R h)(x)\le 2\,K_0\, \R h(x)$.
\end{list}
In particular, it follows from $(c)$ and the definition of $S$
that $\R h\, uv_1^{1/(\frac{r}{\theta})'\delta} \in A_1$ and
therefore $\R h\, uv^{1/(\frac{r}{\theta})'}= \R h\,
uv_1^{1/\delta (\frac{r}{\theta})'}v_2^{1/\delta
(\frac{r}{\theta})'} \in A_\infty$.

To apply the hypothesis \eqref{extrapol2} we must first check that
the left-hand side is finite, but this follows at once from
H\"older's inequality and $(b)$:
\begin{multline*}
\int_{\R^n} f(x)^{\frac{1}{r}}\,\R
h(x)\,u(x)\,v(x)^{1-\frac{\theta}{r}}\,dx  \le
\big\|(f\,v^{-\theta})^{\frac1r}\big\|_{L^{r/\theta,\infty}(uv)}\,\|\R
h\|_{L^{(r/\theta)',1}(uv)}
\\
\le
2\,\big\|f\,v^{-\theta}\big\|_{L^{1/\theta,\infty}(uv)}^{\frac1r}\|h\|_{L^{(\frac{r}{\theta})',1}(uv)}
< \infty.
\end{multline*}
Thus since $\R h\, uv^{1/(\frac{r}{\theta})'} \in A_\infty$ by
\eqref{extrapol2}
\begin{align*}
 \int_{\R^n}
f(x)^{\frac{1}{r}}\,h(x)\,u(x)\,v(x)^{1-\frac{\theta}{r}}\,dx &
\leq \int_{\R^n} f(x)^{\frac{1}{r}}\,\R
h(x)\,u(x)\,v(x)^{1-\frac{\theta}{r}}\,dx
\\
&\leq  C\,\int_{\R^n} g(x)^{\frac{1}{r}}\,\R
h(x)\,u(x)\,v(x)^{1-\frac{\theta}{r}}\,dx \\
& \leq C\,\big\| (g\,v^{-\theta})^\frac{1}{r}
\big\|_{L^{r/\theta,\infty}(uv)} \,
\|\R h\|_{L^{(\frac{r}{\theta})',1}(uv)}\\
& \le  2\,C\,\big\|
g\,v^{-\theta}\big\|_{L^{1/\theta,\infty}(uv)}^{\frac{1}{r}}.
\end{align*}
Since $C$ is independent of $h$, inequality \eqref{extrapolweak}
follows finishing the proof of the theorem.

\section{Proof of Theorem \ref{theor:max}}\label{section:max}

\subsection{Proof of \eqref{max1} }

The following lemma is important in the proof.

\begin{lemma}\label{local}
Let $f$ be a positive and locally integrable function. Then for
 $r>1$ there exists a positive real number $a$ depending on $f$ and
$\lambda $ such that the inequality
\begin{equation*} \left(
\int_{\left| y\right| \le a^{\frac{1}{r-1}}}f(y)dy\right)
a^n=\lambda
\end{equation*}
holds.
\end{lemma}

\begin{proof}
\bigskip Consider the function
\begin{equation*}
g(a)=\left( \int_{\left| y\right| \le
a^{\frac{1}{r-1}}}f(y)dy\right) a^n\text{, for }a\geq 0,
\end{equation*}
then by the hypothesis we have that $g$ is a continuous and non
decreasing function. Furthermore, $g(0)=0 $, and %
$g(+\infty)=+\infty $, and therefore by the mean value theorem
there exists $a$ which satisfies the  conditions of lemma.
\end{proof}

Let $u\ge 0$ and $v(x)=|x|^{-nr}$ with $r>1$. By homogeneity we can
assume that $\lambda=1$. Also, for simplicity we denote $g=fv$.  Now, for each integer $k$ we denote $G_{k}=$ $\left\{ 2^{k}<\left|
x\right| \leq 2^{k+1}\right\} $, $I_{k}=$ $\left\{ 2^{k-1}<\left|
x\right| \leq
2^{k+2}\right\} $, $L_{k}=$ $\left\{ 2^{k+2}<\left| x\right| \right\} $, $%
C_{k}=\left\{ \left| x\right| \leq 2^{k-1}\right\} .$

It will be enough to prove the following estimates

\begin{equation}
\sum_{k\in Z}uv\left\{ x\in
G_{k}: M(g\chi _{I_{k}}) (x) >\frac{1}{\left| x\right|
^{nr}}\right\} \leq C_{r,n} \int g\,Mu,  \label{fac}
\end{equation}

\begin{equation}
\sum_{k\in Z}uv\left\{ x\in
G_{k}: M(g\chi _{L_{k}})(x) >\frac{1 }{\left| x\right|
^{nr}}\right\} \leq C_{r,n}\int g\,Mu,  \label{med}
\end{equation}
\begin{equation}
\sum_{k\in Z}uv\left\{ x\in
G_{k}:M(g\chi _{C_{k}})(x) >\frac{1}{\left| x\right|
^{nr}}\right\} \leq C_{r,n}\int g\,Mu.  \label{dif}
\end{equation}

Taking into account that in $G_{k},$ $ v(x)=\frac{1}{\left|%
x\right| ^{nr}}\sim 2^{-knr}$, using the $(1,1)$ weak type
inequality of $M$ with respect to the pair of weights $(u,Mu)$ and
since the subsets $I_{k}$ overlap at most three times we
obtain $\left( \ref{fac}\right) $.

To prove inequality \eqref{med}  we will estimate
$M(g\chi _{L_{k}}) (x) $. Observe that if $x$ belongs to
$G_{k}$ and $y\in L_{k}=$ $\left\{ 2^{k+2}<\left| y\right|
\right\} ,$ and if $\left| y-x\right| \leq \rho $, we have that
$\frac{\left| y\right| }{2}\leq \rho $,
\begin{equation*}
\frac{1}{\rho ^{n}}\int_{\left| y-x\right| \leq \rho }g(y)\chi
_{L_{k}}\left( y\right) dy\leq C_{n}\int_{2^{k+2}<\left| y\right| }\frac{%
g(y)}{ \left| y\right|^n  }dy \leq C_{n}\int_{|x|<\left| y\right| }\frac{%
g(y)}{ \left| y\right|^n  }dy.
\end{equation*}
If we denote $F(x)= \int_{ |x|<|y| }\frac{%
g(y)}{ \left| y\right|^n  }dy$ the left hand side of \eqref{med}
is bounded by
$$
\sum_{k\in Z}2^{-krn} u\left\{ x\in \mathbb{R}^{n}: F(x)
>C\,2^{-knr} \right\} \approx \int_{0}^{\infty}  t
u\left\{ x\in \mathbb{R}^{n}: F(x)
>t \right\} \frac{dt}{t}
$$
$$
= \int_{\mathbb{R}^{n}} F(x)\,u(x)dx
= \int_{\mathbb{R}^{n}}  \int_{ |x|<|y| }\frac{%
g(y)}{ \left| y\right|^n  }dy   \,u(x)dx
$$
$$
=\int_{\mathbb{R}^{n}} g(y)\,\frac{1}{ \left| y\right|^n  }\int_{
|x|<|y| }   \,u(x)dx\, dy \le C\int_{\mathbb{R}^{n}}
g(y)\,Mu(y)dy.
$$

To prove \eqref{dif} we estimate $M(g\chi _{C_{k}}) (x) $ for $x\in
G_{k}.$ Indeed, if $y\in C_{k}$, $2\left| y\right|
<\left|
x\right| $ and since $%
M(g\chi _{C_{k}})(x) \leq \frac{c_n}{\left| x\right| ^{n}}%
\int_{C_{k}}g(y)dy$, we obtain
\begin{equation*}
M(g\chi _{C_{k}})(x) \leq \frac{C}{\left| x\right| ^{n}}%
\int_{C_{k}}g\leq \frac{C}{\left| x\right| ^{n}}\int_{\left|
y\right| \leq \frac{\left| x\right| }{2}}g,
\end{equation*}
Thus, since the subsets $G_{k}$ are disjoint, the left hand side in \eqref{dif} is bounded by
$$
uv\left\{ x\in \mathbb{R}^{n}:\frac{C}{%
\left| x\right| ^{n}}\int_{\left| y\right| \leq \frac{\left| x\right| }{2}}g>%
\frac{1 }{\left| x\right| ^{nr}}\right\}.
$$

Now, if $a$ denotes the positive real number that appears in Lemma
\ref{local} (i.e., $a$ satisfies
 $1 =\left( \int_{\left| y\right| \leq a^{%
\frac{1}{r-1}}}g\right) a^n$, we express the last integral in the
following way:
\begin{gather}
uv\left( \left\{ x:\frac{C}{\left|
x\right|
^{n}}\int_{\left| y\right| \leq \frac{\left| x\right| }{2}}g>\frac{1 }{%
\left| x\right| ^{nr}}\right\} \right) =uv%
\left( \left\{ \left| x\right| \leq
a^{\frac{1}{r-1}}:\frac{C}{\left|
x\right| ^{n}}\int_{\left| y\right| \leq \frac{\left| x\right| }{2}}g>\frac{%
1 }{\left| x\right| ^{nr}}\right\} \right) +  \label{serie} \\
+\sum_{k=0}^{\infty }uv\left(
\left\{ x:
2^{k}a^{\frac{1}{r-1}}<\left| x\right| \leq 2^{k+1}a^{\frac{1}{r-1}} \,\,\mbox{and}\,\, \frac{C%
}{\left| x\right| ^{n}}\int_{\left| y\right| \leq \frac{\left| x\right| }{2}%
}g>\frac{1 }{\left| x\right| ^{nr}}\right\} \right) \notag
\end{gather}

If $\left| x\right| \leq a^{\frac{1}{r-1}},$ since $\left|
y\right| \leq \frac{\left| x\right| }{2}$ we have that
$\left| y\right| \leq a^{\frac{1}{r-1}}$, thus the set %

$$\left\{ \left| x\right| \leq a^{\frac{1}{r-1}}:%
\frac{C}{\left| x\right| ^{n}}\int_{\left| y\right| \leq
\frac{\left| x\right| }{2}}g>\frac{1 }{\left| x\right|
^{nr}}\right\} \subset \left\{ \left| x\right| \leq
a^{\frac{1}{r-1}}:\left| x\right| ^{n(r-1)}>C\left( \int_{\left|
y\right| \leq a^{\frac{1}{r-1}}}g\right) ^{-1} \right\}. $$

Taking into account the last inclusion and since $\left( \int_{\left|
y\right| \leq a^{\frac{1}{r-1}}}g\right) ^{-1}=a^n$, the first
summand in the second term in $\left( \ref{serie}\right) $ is
bounded by
$$
uv(\{ \left| x\right|
^{r-1}>Ca \} ) =
uv(\{ |x|
>ca^{r'-1}\} ).
$$
Using again Lemma \ref{local}, the last term can be estimated by
$$
\int_{|x|>C\,a^{r'-1}} uv\,dx \le C\,
\sum_{k=1}^{\infty }\frac{ 1  }{(2^ka^{r'-1})^{nr}} \int_{
c2^{k-1}a^{r'-1} \le   |x|< c2^{k}a^{r'-1}} u(x)\,dx \le
$$

$$
\le C\,\sum_{k=1}^{\infty }\frac{ 1  }{2^{k(r-1)n}}\frac{1}{a^n}
\frac{1}{ (c2^ka^{r'-1})^n  }\int_{|x|\le c2^ka^{r'-1}} u(x)\,dx
$$
$$
=C \sum_{k=1}^{\infty }\frac{ 1  }{2^{k(r-1)n}} \int_{|y|\leq
a^{r'-1}} g(y)\,dy \frac{1}{ (c2^ka^{r'-1})^n }\int_{|x|\le
c2^ka^{r'-1}} u(x)\,dx,
$$
and this is bounded by
$$
\le C\sum_{k=1}^{\infty }\frac{ 1  }{2^{k(r-1)n}} \int_{|y|\leq
a^{r'-1}} g(y)Mu(y)\,dy \le C\, \int g\,Mu. $$

To finish, we must estimate the series in \eqref{serie}. It is
clear that sum is bounded by
$$
\sum_{k=0}^{\infty }uv\left(
\left\{  x \in 2^{k}a^{r'-1}<| x| \leq 2^{k+1}a^{r'-1} \right\}
\right) \le C\, \sum_{k=0}^{\infty }\frac{ 1 }{(2^ka^{r'-1})^{nr}}
\int_{ 2^{k-1}a^{r'-1} \le   |x|< 2^{k}a^{r'-1}} u\,dx
$$
and arguing as before we conclude the proof of
$\left(\ref{dif}\right) $.

\begin{remark}\rm
We observe that the proof only uses the following conditions for a
sublinear operator $T$: a) $T$ is of weak type $(1,1)$ with
respect to the pair of weights $(u,Mu)$ and b) $T$ is a convolution
type operator such that the associated kernel satisfies the usual
standard condition:
$$
|K(x)|\le \frac{c}{|x|^n}.
$$
In particular if $u \in A_1$, this observation can be applied to
the usual Calder{\'o}n-Zygmund singular integral operators and
moreover to the strongly singular integral operators (see
\cite{Ch} and \cite{F}).

\end{remark}

\section*{References}

\end{document}